\documentclass[11pt]{amsproc}
\usepackage{mathrsfs}
\usepackage{amsmath}
\usepackage{amssymb}
\usepackage{amsfonts}
\usepackage[pdftex]{graphicx}
\usepackage{epstopdf}

\theoremstyle{plain}
\newtheorem{theorem}{\bf Theorem}[section]
\newtheorem{lemma}[theorem]{\bf Lemma}
\newtheorem{corollary}[theorem]{\bf Corollary}
\newtheorem{proposition}[theorem]{\bf Proposition}
\theoremstyle{definition}
\newtheorem{definition}[theorem]{\bf Definition}

\newtheorem{remark}[theorem]{\bf Remark}
\numberwithin{equation}{section}
\numberwithin{equation}{section}
\numberwithin{equation}{section}
\numberwithin{equation}{section}
\numberwithin{equation}{section}
\numberwithin{equation}{section}
\numberwithin{equation}{section}
\numberwithin{equation}{section}
\numberwithin{equation}{section}
\numberwithin{equation}{section}
\numberwithin{equation}{section}
\numberwithin{equation}{section} 

\newcommand{\R}{\mathbb{R}}

\newcommand{\N}{\mathbb{N}}
\newcommand{\Z}{\mathbb{Z}}

\numberwithin{equation}{section}
\numberwithin{equation}{section}
\numberwithin{equation}{section}
\numberwithin{equation}{section}
\numberwithin{equation}{section}
\numberwithin{equation}{section}
\numberwithin{equation}{section}
\numberwithin{equation}{section}
\numberwithin{equation}{section}
\numberwithin{equation}{section}
\numberwithin{equation}{section}
\numberwithin{equation}{section} 

\def\rp#1{^{\overline{#1}}}

\def\Ceil#1{\left\lceil #1 \right\rceil}

\usepackage{aliascnt}
\def\NewTheorem#1{%
  \newaliascnt{#1}{equation}
  \newtheorem{#1}[#1]{#1}
  \aliascntresetthe{#1}
  \expandafter\def\csname #1autorefname\endcsname{#1}
}

\def\Sum{\sum\limits}

\newcommand{\C}{\mathbb{C}}

\def\Ceil#1{\left\lceil #1 \right\rceil}

\usepackage[numbers,sort&compress]{natbib}
\bibpunct[, ]{[}{]}{,}{n}{,}{,}
\renewcommand\bibfont{\fontsize{10}{12}\selectfont}

\theoremstyle{plain}

\begin{document}


\author[Areeba Ikram]{Areeba Ikram\\
Colorado School of Mines\\ Department of Applied Mathematics \& Statistics\\
1500 Illinois St., Golden, CO}


\title[Lyapunov Inequalities for Nabla Caputo BVPs]{Lyapunov Inequalities for Nabla Caputo Boundary Value Problems}
\thanks{CONTACT Areeba Ikram. Email: aikram@mines.edu}
\maketitle

\begin{center}
This paper is dedicated to my PhD advisor, Allan C. Peterson.
\end{center}

\begin{abstract}
We will establish uniqueness of solutions to boundary value problems involving the nabla Caputo fractional difference under two-point boundary conditions and give an explicit expression for the Green's functions for these problems. Using the Green's functions for specific cases of these boundary value problems, we will then develop Lyapunov inequalities for certain nabla Caputo BVPs.


\end{abstract}


\section{Introduction}


The original Lyapunov inequality from ordinary differential equations includes the following result.

\begin{theorem}\cite{l, kp} \label{lyap} Let $q:[a,b] \to \R$ be continuous. If the boundary value problem 
\begin{equation*} 
\begin{cases}
x'' + q(t) x = 0, & t \in [a,b]\\
 x(a) = x(b) = 0
 \end{cases}
 \end{equation*} has a nontrivial solution,  then $\int_a^b |q(t)| dt > \frac{4}{b-a}.$  
\end{theorem} 

In recent years, due to their many applications for studying solutions to boundary value problems, Lyapunov inqualities have been extended and generalized to BVPs involving fractional operators under various boundary conditions.  Lyapunov inequalities can be used to give existence-uniqueness results for certain nonhomogeneous boundary value problems, study the zeros of solutions, and obtain bounds on eigenvalues in certain eigenvalue problems.  

In the ordinary differential equations case, Lyapunov inequalities for third order linear differential equations with three-point boundary conditions are considered in \cite{fahri}. In fractional order differential equations, a number of recent developments similar to the original Lyapunov inequality have been made; for example, see \cite{dk1, dk2, f1, jrs, js, ma}. Lyapunov inequalities involving the Caputo fractional derivative are studied in \cite{f1,jrs, js, ma}. Fractional equations of order $\alpha$, where $1 < \alpha \leq 2$, are considered in \cite{f1}, \cite{jrs}, and \cite{js} under conjugate, Robin, and  Sturm-Liouville boundary conditions, respectively.  Additionally, \cite{ma} involves fractional equations of order $\alpha$, where $2 < \alpha \leq 3$, and applications of Lyapunov inequalities to a Mittag-Leffler function and an eigenvalue problem are discussed. Boundary value problems involving  the continuous Riemann-Liouville fractional operator of order $\alpha$, where $2 < \alpha \leq 3$, as well as extensions including fractional BVPs with solutions defined on multivariate domains are considered in  \cite{dk1} and \cite{dk2}. A reduction of order technique is used to obtain Lyapunov inequalities in \cite{dk2}, which we will adapt and extend to nabla Caputo BVPs of higher order in this paper.

For fractional difference equations, Lyapunov-type inequalities for two-point conjugate and right-focal boundary value problems involving a delta fractional difference equation of order $\alpha$, where $1 < \alpha \leq 2$, are considered in \cite{f2}.  In \cite{deltaLyap}, Lyapunov inequalities for delta fractional equations are used to study disconjugacy and oscillation of solutions. In the nabla Riemann-Liouville case, a Lyapunov inequality for a boundary value problem of order $\alpha$, where $2 < \alpha \leq 3$, is given in \cite{rightfi}.  Much remains to be explored in Lyapunov inequalities for fractional difference operators, and we will develop some results for the nabla Caputo case in Section 4.

Green's functions play an essential role in deriving Lyapunov inequalities for boundary value problems. A general method of obtaining Lyapunov inequalities involves converting a given boundary value problem to an equivalent integral equation involving an appropriate Green's function and then using bounds on the Green's function \cite{sequentialFerr}.  

This paper is organized as follows. In Section 2, we will give preliminary definitions and results involving the nabla Caputo fractional difference. In Section 3, we will develop Green's functions for BVPs involving the nabla Caputo difference operator. The main results of Section 3 are given in Theorems \ref{GreensfunctionthmkNk} and \ref{Gdeterminant}, which give an explicit form for the unique solutions to the given BVPs. In Section 4, we will develop Lyapunov inequality results using a particular case of the Green's function results from Section 3. The main result of Section 4 is given in Theorem \ref{higherordersubsituting}, which gives Lyapunov inequalities obtained by using the earlier mentioned reduction of order technique. We will end with a corollary which gives sufficient conditions for certain nonhomogenous BVPs to have unique solutions.

\section{Preliminaries}
In this paper, functions will be defined on either of the domains $\N_a:= \{a, a+1, a+2, \hdots \} \text{ or } \N_a^b := \{a, a+1, \hdots, b\},$
where $a, b \in \R$ such that $b-a$ is a positive integer. We will let $\N:=\N_1$. For more details on the background presented in this section, see \cite{gp}.

\begin{definition} The \textbf{nabla difference} of a function $f: \N_{a} \to \R$ is defined by \[ \nabla f(t) := f(t) - f(t-1), \,\ \text{ for } t \in \N_{a+1}.\]  We define nabla differences of any higher order $N \in \N$ recursively; i.e., $\nabla^N f(t):= \nabla (\nabla^{N-1} f) (t), \,\ \text{ for } t \in \N_{a+N}.$ Additionally, we take by convention $\nabla^0 f(t):= f(t).$
\end{definition}

The next proposition gives a binomial formula for the $N$-th order nabla difference.

\begin{proposition} \label{binomialnablaexpansion} Let $f: \N_a \to \R$ and $N \in \N$. Then, 
$$\nabla^N f(t) = \Sum_{i=0}^N(-1)^i {N\choose{i}} f(t-i),$$ for $t \in \N_{a+N}$.
\end{proposition}

\begin{definition}We define the \textbf{backward jump operator}, $\rho: \N_a \to \N_a$, by $\rho(t) := \max\{a, t-1\}$. 
\end{definition}

\begin{definition}\cite[p. 333]{pb} The \textbf{nabla definite integral} of a function $f: \N_{a+1}^b \to \R$, for $c, d \in \N_a^b$, is defined by
$\int_c^d f(t) \nabla t := \begin{cases} \sum\limits_{t=c+1}^d f(t), &  d >c \\
0, &   d = c \\
-\sum\limits_{t=d+1}^c f(t), & d < c.\end{cases}$ \end{definition}

\begin{theorem} \textbf{(Fundamental Theorem of Nabla Calculus)} \label{ftnc}
If $f: \N_{a+1}^b \to \R$ and $F$ is any nabla antidifference of $f$ on $\N_a^b$ (i.e., $\nabla F(t) = f(t)$, for $t \in \N_{a+1}^b$), then
\[ \int_a^b f(t) \nabla t = F(b) - F(a). \]
\end{theorem}

Next, we will define nabla fractional sums and differences. Let $\Z_{\leq 0}$ denote the set of nonpositive integers.

\begin{definition} For $t, r \in \C$, the \textbf{generalized rising function} is defined by \[ t \rp {r} := \begin{cases} \frac{\Gamma(t+r)}{\Gamma(t)}, & \text{ if } t+r, t \not \in \Z_{\leq 0} \\
0, & \text{ if }  t+r \not \in \Z_{\leq 0} \text{, and } t  \in \Z_{\leq 0}\\
(-1)^{r} \frac{(-t)!}{(-t-r)!} & \text{ if } t+r, t \in \Z_{\leq 0}\\
\text{undefined} & \text{ if } t+r  \in \Z_{\leq 0}\text{, and } t \not \in \Z_{\leq 0},   \end{cases} \]
where $\Gamma$ is the Gamma function.
 \end{definition}

\begin{proposition} \label{gammaprop} For $z \in \C \setminus \{0, -1, -2, -3, \hdots \}$, we have $\Gamma(z+1) = z \Gamma(z).$ \end{proposition}


\begin{definition} \label{Taylor}
For $\nu \in \R$, the $\nu$-th order \textbf{nabla Taylor monomial}, based at $s \in \N_a$, is defined for $t \in \N_a$ by
\begin{equation*} 
H_{\nu}(t,s):= \frac{(t-s)\rp{\nu}}{\Gamma(\nu+1)}.
\end{equation*}
\end{definition}

Next, we state several properties of the nabla Taylor monomials.

\begin{theorem} \label{taylormonprops} For $t \in \N_a$ and $\mu \in \R$,
\begin{enumerate}
\item for $\mu \not = 0$, $H_{\mu}(a,a) = 0$ and $H_0(t,a) \equiv 1$;
\item $\nabla H_{\mu}(t,a) = H_{\mu-1}(t,a);$
\item for $\mu \not = -1$, $\int_a^t H_{\mu}(s,a) \nabla s = H_{\mu+1}(t,a);$
\item for $\mu \not = -1$, $\int_a^t H_{\mu}(t,\rho(s)) \nabla s = H_{\mu+1}(t,a);$
\item for $k \in \N_1$, $s \in \{ a + n \mid n \in \Z\},$ and $t \in \N_{s+k+1}$, $H_{-k}(t,s) = 0;$
\end{enumerate}
provided the expressions above are defined.
\end{theorem}

\begin{definition} \label{fracsum} 
Let $f: \N_{a+1} \to \R$ and $\nu > 0$. Then, the \textbf{nabla fractional sum} of $f$ of order $\nu$, based at $a$, is defined by
\begin{equation} \label{fracsumform}
\nabla_a^{-\nu} f(t) := \int_a^t H_{\nu-1}(t, \rho(s)) f(s) \nabla s,
\end{equation}
for $t \in \N_{a+1}$. Also, we define $\nabla_a^{-0} f(t):= f(t)$.
\end{definition}

\begin{definition} \label{fracdif}
Let $f: \N_{a-N+1} \to \R$, $\nu > 0$, and $N:= \Ceil{\nu}$. Then, the $\nu$-th order \textbf{nabla Caputo fractional difference} of $f$ is defined by 
\begin{equation} \label{fracdifform}
\nabla_{a*}^{\nu} f(t) := \nabla_a^{-(N-\nu)}\nabla^N f(t),
\end{equation}
for $t \in \N_{a+1}$. By convention, $\nabla_{a*}^{\nu} f(t) = 0$ for $t \in \{a-k \mid k \in \N_0\}.$
\end{definition}


A variation of constants formula for a nabla Caputo initial value problem is given in the next theorem. 

\begin{theorem} \label{caputoIVPthm} Consider the IVP
\begin{equation} \label{caputoIVP}
\begin{cases}
\nabla_{a*}^{\nu} x(t) = h(t), &t \in \N_{a+1} \\
\nabla^k x(a) = c_k, &k \in \N_0^{N-1},
\end{cases} 
\end{equation}
where $\nu > 0$, $N:= \Ceil{\nu},$ $h:\N_{a+1} \to \R$, and $c_k \in \R$ for $k \in \N_0^{N-1}$. 
Then, the unique solution to the IVP \eqref{caputoIVP} is given by $$x(t) = \sum\limits_{k=0}^{N-1} H_k(t,a) c_k + \nabla_a^{-\nu} h(t), \,\ t \in \N_{a-N+1}. $$
\end{theorem}

We will also state the following Leibniz formula, which is useful when showing that integral expressions satisfy nabla difference equations.

\begin{theorem} \label{Leibniz}\textbf{(Nabla Leibniz Formula).} Assume $f: \N_a \times \N_{a+1} \to \R$. Then, for $t \in \N_{a+1}$, 
\begin{equation} \label{leibnizformula}
\nabla \left( \int_a^t f(t,\tau)\nabla \tau\right) = \int_a^t \nabla_t f(t,\tau) \nabla \tau + f(\rho(t), t). 
\end{equation}
\end{theorem}


\section{Green's Functions}

In this section we will develop Green's functions for $(k, N-k)$ BVPs involving the nabla Caputo difference operator.
The next remark motivates the theorem that follows, which will  establish a form for a general solution to $\nabla_{a*}^{\nu} x(t) = h(t)$ in terms of nabla Taylor monomials based at modified points. This form will be useful when considering $(k, N-k)$ boundary value problems.

\begin{remark} In the continuous case, for each $p \in \N_1^{n-1}$, $x_p(t):=\frac{(t-a)^p}{p!}$  is a solution to the equation $x^{(n)}=0$ satisfying the initial conditions $x^{(i)}(a) = 0$ for $i \in \N_0^{p-1}$. In particular, we say $x_p(t)$ has a zero of multiplicity $p$ at $t=a$. In an analogous way,  for each $p \in \N_1^{N-1}$, $H_p(t,a-N+p)$ satisfies $\nabla^ix(a-N+p)=0$ for $i \in \N_0^{p-1}$ and has $p$ consecutive zeros on the domain $\N_{a-N+1}$ at $t=a-N+1, \hdots, a-N+p$.
\end{remark}


\begin{theorem} \label{caputogensolntaylor} Let $\nu > 0$ and $N:= \Ceil{\nu}$. A general solution to 
\begin{equation} \label{caputogensolntayloreqn}
\nabla_{a*}^{\nu} x(t) = h(t), \,\ t \in \N_{a+1}
\end{equation} is given by 
\begin{equation}\label{caputogensolntaylorsoln}
 x(t) = \Sum_{p=0}^{N-1} c_p H_p(t,a - N +p) + \nabla_a^{-\nu} h(t), 
\end{equation}
for $t \in \N_{a-N+1}$, where $c_p$ for $p \in \N_0^{N-1}$ are arbitrary constants. 
\end{theorem}
\begin{proof} Note that, for $t \in \N_{a+1}$ and $p \in \N_0^{N-1}$, \begin{align} \label{applycaputototaylormon}
\nabla_{a*}^{\nu}H_p(t,a-N+p)  &\overset{\eqref{fracdifform}}{=} \nabla_a^{-(N-\nu)} \nabla^N H_p(t,a-N+p) \nonumber \\
&=  0,
\end{align} by repeated applications of  Theorem \ref{taylormonprops}, part (2) and by Theorem \ref{taylormonprops}, part (5). Also, \begin{align} \label{hparticularsoln}
\nabla_{a*}^{\nu} \nabla_a^{-\nu} h(t) &{=}  h(t),
\end{align}
for $t \in \N_{a+1}$ since by Theorem \ref{caputoIVPthm}, $\nabla_a^{-\nu} h(t)$ is the unique solution to $\nabla_{a*}^{\nu} x(t) = h(t)$ satisfying the initial conditions $\nabla^k x(a) = 0$, $k \in \N_0^{N-1}$. Hence, by \eqref{applycaputototaylormon}, \eqref{hparticularsoln}, and linearity of the operator $\nabla_{a*}^{\nu}$, we have that $x(t)$, given by \eqref{caputogensolntaylorsoln}, is a solution to \eqref{caputogensolntayloreqn} and is defined on $\N_{a-N+1}$. 
 
Now suppose $y(t)$ is any solution to $\nabla_{a*}^{\nu} y(t) = 0$. Then, $y(t)$ is determined by its initial values, $\nabla^k y(a)$ for $k \in \N_0^{N-1}$, by Theorem \ref{caputoIVPthm}. Let the initial values of $H_p(t,a-N+k)$ for $p \in \N_0^{N-1}$ be given by the vector $\mathbf{v}_p:=$
\begin{equation} \label{ICsvector}
\langle H_p(t,a-N+k)|_{t=a}, \nabla H_p(t,a-N+k)|_{t=a}, \hdots, \nabla^{N-1}H_p(t,a-N+k)|_{t=a}\rangle. 
\end{equation}Then, for $k \in \N_0^{N-1}$,
\begin{align*} 
\nabla^k H_p(t,a-N+p)\Big|_{t=a} 
&=\begin{cases} \frac{(N-p)\rp{p-k}}{(p-k)!} = \frac{(N-k-1)!}{(N-p-1)!(p-k)!}, & k \leq p \\
0, & k > p, \end{cases}
\end{align*} by Theorem \ref{taylormonprops}, parts (2) and (5).  Therefore, the vectors $\mathbf{v}_0, \mathbf{v}_1, \hdots, \mathbf{v}_{N-1} \in \R^{N}$ are linearly independent. It follows that $y(t) = \Sum_{p=0}^{N-1} c_p H_p(t,a-N+p)$ for some $c_p \in \R$, for $p \in \N_0^{N-1}$.

Next, suppose $w(t)$ is a solution to \eqref{caputogensolntayloreqn}. It follows that $w(t) - \nabla_a^{-\nu}h(t)$ is a solution to $\nabla_{a*}^{\nu} y(t) = 0.$ Then, by the above argument, $w(t) - \nabla_a^{-\nu} h(t) = \Sum_{p=0}^{N-1}c_p H_p(t,a-N+p)$ for some $c_p \in \R$. Thus, $w(t) =  \Sum_{p=0}^{N-1}c_p H_p(t,a-N+p) + \nabla_a^{-\nu} h(t)$, so \eqref{caputogensolntaylorsoln} gives a general solution to \eqref{caputogensolntayloreqn}.
\end{proof}

Next, we get a form of any solution $x(t)$ to the homogeneous equation $\nabla_{a*}^{\nu}x(t) = 0$ which satisfies $k$ homogeneous initial conditions, for any fixed $k \in \N_1^{N-1}$, where $N:= \Ceil{\nu}$.

\begin{lemma} \label{LeftBCssimplifysoln}
Let $\nu > 1$, $N:= \Ceil{\nu}$, $k \in \N_1^{N-1}$, and suppose $x: \N_{a-N+1} \to \R$ is a solution to the equation
\begin{equation} \label{homeqnsolnhomics}
\nabla_{a*}^{\nu} x(t) = 0, \,\ t \in \N_{a+1}.
\end{equation} Moreover, assume that $x$ satisfies the conditions
$\nabla^i x(a-N+k) = 0, \,\ i \in \N_0^{k-1}.$ Then,
$x(t) =  \Sum_{p=k}^{N-1} c_p H_p(t,a-N+p), $
where $c_k, c_{k+1}, \hdots, c_{N-1} \in \R$.
\end{lemma}
\begin{proof} Let $x(t)$ be a solution to \eqref{homeqnsolnhomics}. Then, by Theorem \ref{caputogensolntaylor}, we have 
$x(t) = \Sum_{p=0}^{N-1} c_p H_p(t,a-N+p), \,\ t \in \N_{a-N+1}, $
where $c_p$ for $p \in \N_0^{N-1}$ are constants. Let $i \in \N_0^{k-1}$ and consider 
\begin{align}\label{plugginginics}
\nabla^i x(a-N+k) &= \Sum_{p=0}^{N-1} c_p \nabla^iH_p(t,a-N+p)\mid_{t=a-N+k} \nonumber\\
&= \Sum_{p=i}^{N-1} c_p H_{p-i}(a-N+k,a-N+p).
\end{align}
Note that for $p \geq i$,
$
H_{p-i}(a-N+k, a-N+p) 
=\begin{cases} 0, & k-p \leq 0 \\
\frac{(k-i-1)!}{(p-i)!(k-p-1)!}, & k-p > 0.
\end{cases}$
 From \eqref{plugginginics} and $\nabla^i x(a-N+k) = 0$, for each  $i \in \N_0^{k-1}$, we have
\begin{equation}\label{pluginsimplify}
\Sum_{p=i}^{k-1} c_p \frac{(k-i-1)!}{(p-i)!(k-p-1)!} = 0. 
\end{equation}
Letting $i=k-1, \,\ k-2, \,\ \hdots, \,\ 0$ in \eqref{pluginsimplify} with the given order implies $c_{k-1}= c_{k-2} = \cdots = c_0 = 0$, respectively. Hence, $x(t) = \Sum_{p=k}^{N-1} c_p H_p(t,a-N+p)$. \end{proof}


Next, we give an existence-uniqueness result, often referred to as Fredholms Alternative Theorem \cite{greens},  for two-point boundary value problems involving the operator $\nabla_{a*}^{\nu}$.  

\begin{theorem} (Existence-Uniqueness Theorem) \label{existenceuniquenessbvpkNk}
Let $\nu > 1$, $N:= \Ceil{\nu}$, $k \in \N_1^{N-1}$, and $h:\N_{a+1}^b \to \R$. Furthermore, let $j_m \in \N_0^{N-1}$ for  $m \in \N_1^{N-k},$ with
$j_1 < j_2 < j_3 < \cdots < j_{N-k}$, and assume $b-a \in \N_{\max\{1,j_{N-k}-N+k+1\}}$. Then, the homogeneous $(k, N-k)$ BVP
\begin{equation} \label{homogeneouskNk}
\begin{cases}
\nabla_{a*}^{\nu}y(t) = 0,  &t \in \N_{a+1}^b \\
\nabla^i y(a-N+k) = 0, &i \in \N_0^{k-1} \\
\nabla^{j_m} y(b) = 0, &m \in \N_1^{N-k} 
\end{cases}
\end{equation}
has only the trivial solution if and only if the nonhomogeneous $(k, N-k)$ BVP
\begin{equation} \label{nonhomogeneouskNk}
\begin{cases}
\nabla_{a*}^{\nu}w(t) = h(t), &t \in \N_{a+1}^b \\
\nabla^i w(a-N+k) = A_i, &i \in \N_0^{k-1} \\
\nabla^{j_m} w(b) = B_{j_m}, &m \in \N_1^{N-k},
\end{cases}
\end{equation}
has a unique solution for each $A_i, B_{j_m} \in \R$, for $i \in \N_0^{k-1}$ and $m \in \N_1^{N-k}$.
\end{theorem}
\begin{proof} By Theorem \ref{caputogensolntaylor}, a general solution to $\nabla_{a*}^{\nu} y(t) = 0$ is given by 
$y(t) = c_0 H_0(t,a-N) + c_1H_1(t,a-N+1) + \cdots + c_{N-1} H_{N-1}(t,a-1). $
Fix $k \in \N_1^{N-1}$, let $\alpha:=a-N+k$, and let $x_p(t):= H_p(t,a-N+p)$. Then, $y$ satisfies the boundary conditions in \eqref{homogeneouskNk} if and only if the vector equation 
\[\underbrace{ \left(\begin{matrix}x_0(\alpha) & x_1(\alpha) & \cdots &  x_{N-1}(\alpha)  \\
\nabla x_0(\alpha) & \nabla x_1(\alpha) & \cdots & \nabla x_{N-1}(\alpha)  \\
\vdots & \vdots & \ddots & \vdots \\
 \nabla^{k-1} x_0(\alpha) & \nabla^{k-1} x_1(\alpha) &\cdots &  \nabla^{k-1} x_{N-1}(\alpha) \\
\nabla^{j_1} x_0(b) & \nabla^{j_1} x_1(b) & \cdots &  \nabla^{j_1} x_{N-1}(b)  \\
\nabla^{j_2} x_0(b) &  \nabla^{j_2} x_1(b) &\cdots &\nabla^{j_2} x_{N-1}(b)  \\
\vdots & \vdots & \ddots & \vdots \\
 \nabla^{j_{N-k}} x_0(b) &  \nabla^{j_{N-k}} x_1(b) & \cdots &  \nabla^{j_{N-k}} x_{N-1}(b) \end{matrix}\right)}_{=:M} \underbrace{\left(\begin{matrix} c_0 \\
 c_1 \\
 \vdots \\
 c_{k-1} \\
 c_k \\
 c_{k+1} \\
 \vdots \\
 c_{N-1} \end{matrix} \right)}_{=:\textbf{c}}  = \left(\begin{matrix} 0 \\
 0 \\
 \vdots \\
 0 \\
 0 \\
 0 \\
 \vdots \\
 0 \end{matrix}\right) \] holds. Since, by hypothesis, the homogeneous BVP \eqref{homogeneouskNk} has only the trivial solution, the above vector equation has only the trivial solution $\textbf{c}=\textbf{0}$. Hence, $\det M \not = 0$. 

Now suppose $w$ is a solution to the nonhomogeneous equation $\nabla_{a*}^{\nu} w(t) = h(t)$. Then, by Theorem \ref{caputogensolntaylor}, we have $w(t) = d_0 H_0(t,a-N) + d_1 H_1(t,a-N+1) + \cdots + d_{N-1} H_{N-1}(t,a-1) + \nabla_{a}^{-\nu} h(t),$ for some constants $d_0, d_1, \hdots, d_{N-1}$. Then, the boundary value problem \eqref{nonhomogeneouskNk} has a solution if and only if the vector equation
\[M \underbrace{\left(\begin{matrix} d_0 \\ d_1 \\ \vdots \\ d_{N-1} \end{matrix} \right)}_{:=\textbf{d}} = \left( \begin{matrix} A_0 - \nabla_a^{-\nu}h(a-N+k) \\ \vdots \\ A_{k-1} - \nabla^{k-1} \nabla_a^{-\nu}h(a-N+k) \\ B_{j_1} - \nabla^{j_1}\nabla_a^{-\nu}h(b) \\ \vdots \\ B_{j_{N-k}} - \nabla^{j_{N-k}}\nabla_a^{-\nu}h(b) \end{matrix} \right)\] has a solution. Since $\det M \not = 0$,  this vector equation has a unique solution $\textbf{d}$, so the BVP \eqref{nonhomogeneouskNk} has a unique solution.

The proof of the converse is straightforward and hence omitted.
\end{proof}

 Let $\alpha:= a-N+k$. In the remainder of this section, we let $D:=$
 \begin{equation}\label{nonzeroN-kdet}
 \left( \begin{matrix} \nabla^{j_1}H_k(b,\alpha) & \nabla^{j_1} H_{k+1}(b,\alpha+1) & \cdots & \nabla^{j_1} H_{N-1}(b,a-1) \\
   \nabla^{j_2}H_k(b,\alpha) & \nabla^{j_2} H_{k+1}(b,\alpha+1) & \cdots & \nabla^{j_2} H_{N-1}(b,a-1) \\
   \vdots & \vdots & \ddots & \vdots \\
   \nabla^{j_{N-k}}H_k(b,\alpha) & \nabla^{j_{N-k}} H_{k+1}(b,\alpha+1) & \cdots & \nabla^{j_{N-k}} H_{N-1}(b,a-1) \end{matrix} \right).
 \end{equation}

 \begin{theorem}\label{necsuffN-kdetthm} A necessary and sufficient condition for uniqueness of solutions to the nonhomogeneous BVP \eqref{nonhomogeneouskNk} is 
$\det D \not = 0$, where $D$ is given by \eqref{nonzeroN-kdet}.
 \end{theorem}
 \begin{proof} By Lemma \ref{LeftBCssimplifysoln}, a solution to $\nabla_{a*}^{\nu}x(t) = 0,$ for $ t \in \N_{a+1}^b$, which satisfies the conditions $\nabla^ix(a-N+k)=0,$ for $i \in \N_0^{k-1}$, where $k \in \N_1^{N-1}$ is fixed, is given by $x(t) = c_k H_k(t,a-N+k) + c_{k+1}H_{k+1}(t,a-N+k+1) + \cdots + c_{N-1} H_{N-1}(t,a-1).$ Using the boundary conditions at $t = b$ in \eqref{homogeneouskNk} in the last equation, we get the vector equation
\[ D \left( \begin{matrix} c_k \\ c_{k+1} \\ \vdots \\ c_{N-1} \end{matrix} \right) = \left(\begin{matrix} 0\\
  0\\ \vdots \\ 0\end{matrix} \right),\] where $D$ is given by \eqref{nonzeroN-kdet}. This vector equation has only the trivial solution if and only if $\det D \not = 0$. It follows by Theorem \ref{existenceuniquenessbvpkNk} that the nonhomogeneous BVP \eqref{nonhomogeneouskNk} has a unique solution if and only if $\det D \not = 0$.
  \end{proof}

The next two lemmas are used to show that $\det D \not = 0$.
\begin{lemma} \label{determinantnonzeroequivalence} Let $D$ be as in \eqref{nonzeroN-kdet}. Then, $\det D \not = 0$ if and only if $\det \hat{D} \not = 0$, where $\hat{D}:=$
 \begin{equation}\label{nonzeroN-kdetAppendixtransformed}
 \left( \begin{matrix} \prod\limits_{i=k+1}^{N-1}(i- j_1) & \prod\limits_{i=k+2}^{N-1}(i- j_1)   & \cdots & \prod\limits_{i=N-1}^{N-1} (i- j_1)  & 1 \\
 \prod\limits_{i=k+1}^{N-1}(i- j_2) & \prod\limits_{i=k+2}^{N-1}(i- j_2)   & \cdots & \prod\limits_{i=N-1}^{N-1}(i- j_2) & 1 \\
   \vdots & \vdots & \ddots & \vdots \\
 \prod\limits_{i=k+1}^{N-1}(i- j_{N-k}) & \prod\limits_{i=k+2}^{N-1}(i- j_{N-k})   & \cdots & \prod\limits_{i=N-1}^{N-1}(i- j_{N-k}) & 1 \\ \end{matrix} \right).
 \end{equation}
\end{lemma}
\begin{proof} Let $p \in \N_{k}^{N-1}$ and $m \in \N_1^{N-k}$. Then, the entry in row $m$ and column $p-k+1$ of the matrix $D$ is $H_{p-j_{m}}(b, a-N+p) = \frac{\Gamma(b-a+N-j_m)}{\Gamma(b-a+N-p)\Gamma(p-j_m+1)}.$   Then,
 $ \det D = \frac{\Gamma(b-a+N-j_1)\Gamma(b-a+N-j_2)\cdots \Gamma(b-a+N-j_{N-k})}{\Gamma(b-a+N-k)\Gamma(b-a+N-k-1)\cdots\Gamma(b-a+1)} \det E, $
 where the entry in row $m$ and column $p-k+1$ of the matrix $E$ is $  \frac{1}{\Gamma(p-j_m+1)}.$
 Note  $\det E \not = 0$ if and only if $\det D \not = 0$. Next, multiplying row $m$ of the matrix $E$ by $\Gamma(N-j_m)$ for each $m \in \N_1^{N-k}$ and then using the property of the Gamma function given in Proposition \ref{gammaprop}, we obtain the matrix
 $\hat{D}$. Moreover, $\det \hat{D} \not = 0$ if and only if $\det D \not = 0$. 
\end{proof}

It can be shown that the matrix $\hat{D}$ can be obtained by elementary column operations on the matrix $E$ defined by\[\left(\begin{matrix} (-1)^{N-k-1}(j_1)^{N-k-1} & (-1)^{N-k-2}(j_1)^{N-k-2} & \cdots & (-1)j_1 & 1 \\
(-1)^{N-k-1}(j_2)^{N-k-1} & (-1)^{N-k-2}(j_2)^{N-k-2} & \cdots & (-1)j_2 & 1 \\
\vdots & \vdots & \ddots & \vdots \\
(-1)^{N-k-1}(j_{N-k})^{N-k-1} & (-1)^{N-k-2}(j_{N-k})^{N-k-2} & \cdots & (-1)j_{N-k} & 1 \end{matrix} \right). \] It follows by the Vandermonde determinant formula \cite[p. 17]{vandermonde} and properties of determinants that $\det E = (-1)^{\frac{(N-k)(N-k-1)}{2}}\prod\limits_{1 \leq p < r \leq N-k} (j_p-j_r) \not = 0$ since $j_1 < j_2 < \cdots < j_{N-k}$. Hence, we get the next lemma.

\begin{lemma} \label{determinantnonzeropolynomial}
Let $k \in \N_1^{N-1}$ be fixed, $j_1 < j_2 < \cdots < j_{N-k}$,  and $j_m \in \N_0^{N-1}$, for $m \in \N_1^{N-k}$. Then, $\det \hat{D} \not = 0$, where $\hat{D}$ is given by \eqref{nonzeroN-kdetAppendixtransformed}.
\end{lemma}

The next theorem follows directly from Lemmas \ref{determinantnonzeroequivalence} and \ref{determinantnonzeropolynomial}. 

\begin{theorem} \label{Dnotzero} The matrix $D$, given by \eqref{nonzeroN-kdet}, has a nonzero determinant.
 \end{theorem}
 
Using Theorem \ref{necsuffN-kdetthm} and Theorem \ref{Dnotzero}, we get the next theorem.

\begin{theorem} \label{uniquebvpsolndetDnonzero}  The nonhomogeneous BVP \eqref{nonhomogeneouskNk} has a unique solution.
\end{theorem}

The function $G: \N_{a-N+1}^b \times \N_{a+1}^b \to \R$ given in the next theorem is called the \textit{Green's function} for the homogeneous BVP \eqref{homogeneouskNk}. Note that the Green's function is used to find the unique solution to the nonhomogeneous BVP \eqref{nonhomogeneouskNk}.

\begin{theorem} \label{GreensfunctionthmkNk} Let $\nu > 1$ and $N:=\Ceil{\nu}$. Assume $k \in \N_{1}^{N-1}$,  $j_m \in \N_0^{N-1}$ for  $m \in \N_1^{N-k},$ with $j_1 < j_2 < \cdots < j_{N-k}$, and \\ $b-a \in \N_{\max\{1,j_{N-k}-N+k+1\}}$. For each fixed $s \in \N_{a+1}^b$, let $u(t,s)$ be defined as the solution to the BVP
\begin{equation} \label{ubvp}
\begin{cases}
\nabla_{a*}^{\nu} u(t,s) = 0, &t \in \N_{a+1}^b\\
\nabla^i u(a-N+k,s) = 0, &i \in \N_0^{k-1} \\
\nabla^{j_m} u(b,s) = -\nabla^{j_m} H_{\nu-1}(b,\rho(s)), &m \in \N_1^{N-k}.
\end{cases}
\end{equation}
Define \begin{equation} \label{GreenskNk}
G(t,s):= \begin{cases} u(t,s), \text{ if } t \leq \rho(s) \\
v(t,s), \text{ if } t \geq \rho(s),
\end{cases}
\end{equation}
where $v(t,s):= u(t,s) + H_{\nu-1}(t, \rho(s))$ and $(t,s) \in \N_{a-N+1}^b \times \N_{a+1}^b$. Then,
$w(t) := \int_a^b G(t,s) h(s) \nabla s$ is the unique solution to the nonhomogeneous $(k, N-k)$ BVP \eqref{nonhomogeneouskNk} with $A_i, B_{j_m} = 0$, for $i \in \N_0^{k-1}$ and $m \in \N_1^{N-k}$. 
\end{theorem}
\begin{proof} By Theorem \ref{uniquebvpsolndetDnonzero}, the BVP \eqref{ubvp}, for each fixed $s \in \N_{a+1}^b$, has a unique solution, so $u(t,s)$ is well defined. Let $G(t,s)$ be defined as in \eqref{GreenskNk} and $w(t):=\int_a^b G(t,s) h(s)\nabla s$.
First, for $t \in \N_{a-N+1}^b$, 
\begin{align*}
w(t) 
&=\int_a^{t} v(t,s) h(s) \nabla s + \int_{t}^b u(t,s) h(s) \nabla s\\
&=\int_a^b u(t,s) h(s) \nabla s + \int_a^{t} H_{\nu-1}(t,\rho(s)) h(s) \nabla s \\
&\overset{\eqref{fracsumform}}{=}\int_a^b u(t,s) h(s) \nabla s + \nabla_a^{-\nu} h(t).
\end{align*}

For $t \in \N_{a+1}$,
\begin{align*}
\nabla_{a*}^{\nu} w(t) & = \nabla_{a*}^{\nu}\left[\int_a^b u(t,s) h(s) \nabla s + \nabla_a^{-\nu} h(t)\right] \\
&= \Sum_{s=a+1}^b \nabla_{a*}^{\nu} u(t,s) h(s) + \nabla_{a*}^{\nu} \nabla_{a}^{-\nu} h(t) \\
&\overset{\eqref{ubvp}, \hspace{1pt}\eqref{hparticularsoln}}{=} h(t).
\end{align*}
Since $\nabla_a^{-\nu} h(a-N+1) = \cdots = \nabla_a^{-\nu} h(a) = 0$ by convention, in particular, we get $\nabla^i(\nabla_a^{-\nu}h)(a-N+k) = 0$ for $ i \in \N_0^{k-1}$. Thus, for $i \in \N_0^{k-1}$,
\begin{align*}
\nabla^i w(t)|_{t=a-N+k}&= \int_a^b \nabla^i u(a-N+k,s) h(s) \nabla s + \nabla^i(\nabla_a^{-\nu} h)(a-N+k)\\
 &\overset{\eqref{ubvp}}{=} 0.
\end{align*} 
Moreover, for $j_m \in \N_0^{N-1}$, $m \in \N_1^{N-k}$,
\begin{align*}
\nabla^{j_m} w(t)|_{t=b} &= \int_a^b \nabla_t^{j_m} u(t,s) h(s) \nabla s\Big|_{t=b} + \nabla^{j_m}\left[ \int_a^t H_{\nu-1}(t,\rho(s)) h(s) \nabla s \right]\Big|_{t=b} \\
&\overset{\eqref{leibnizformula}}{=} \int_a^b \nabla_t^{j_m}u(t,s)h(s) \nabla s\Big|_{t=b}   \\
&  + \left[\nabla^{j_m-1}\int_a^t \nabla_t H_{\nu-1}(t,\rho(s))h(s) \nabla s + H_{\nu-1}(\rho(t), \rho(t)) h(t) \right]\Big|_{t=b} \\
 &= \int_a^b \nabla_t^{j_m}u(t,s)h(s) \nabla s\Big|_{t=b} \\
 &+  \left[\nabla^{j_m-1}\int_a^t \nabla_t H_{\nu-1}(t,\rho(s))h(s)\nabla s\right]\Big|_{t=b} \\
& \quad \vdots \\
&\overset{\eqref{leibnizformula}}{=}\int_a^b \nabla_t^{j_m} u(t,s)h(s) \nabla s \Big|_{t=b}  \\
& + \left[ \int_a^t \nabla_t^{j_m} H_{\nu-1}(t,\rho(s))h(s) \nabla s + \nabla_t^{j_m-1}H_{\nu-1}(\rho(t), \rho(t)) h(t) \right]\Big|_{t=b} \\
&\overset{\eqref{ubvp}}{=}\int_a^b -\nabla_t^{j_m}H_{\nu-1}(b,\rho(s)) h(s) \nabla s  \\
&+ \int_a^b \nabla_t^{j_m} H_{\nu-1}(b, \rho(s)) h(s) \nabla s \\ 
&= 0.
\end{align*} 
\end{proof}

The proof of the following corollary is standard and follows in a straightforward manner from Theorem \ref{GreensfunctionthmkNk}.

\begin{corollary} \label{solntononhomkNkaddw} Assume that the hypotheses of Theorem \ref{GreensfunctionthmkNk} hold. Also, let $h: \N_{a+1}^b \to \R$, $G(t,s)$ be as defined in \eqref{GreenskNk}, and $w$ be the unique solution to the BVP 
\begin{equation*} 
\begin{cases} \nabla_{a*}^{\nu} w(t) = 0,  &t \in \N_{a+1}^b \\
\nabla^i w(a-N+k) = A_i,  &i \in \N_0^{k-1} \\
\nabla^{j_m}w(b) = B_{j_m},  &m \in \N_1^{N-k}.
\end{cases} \end{equation*} Then, the unique solution to the nonhomogeneous BVP
\[\begin{cases} \nabla_{a*}^{\nu} y(t) = h(t), &t \in \N_{a+1}^b \\
\nabla^i y(a-N+k) = A_i, &i \in \N_0^{k-1} \\
\nabla^{j_m}y(b) = B_{j_m}, &m \in \N_1^{N-k},
\end{cases}\] 
is given by $y(t):= w(t) + \int_a^b G(t,s) h(s) \nabla s. $
\end{corollary}

\begin{theorem} \label{Gdeterminant}
Assume that the hypotheses of Theorem \ref{GreensfunctionthmkNk} hold. Then, the Green's function for the $(k, N-k)$ BVP \eqref{homogeneouskNk} is given by \eqref{GreenskNk}, where $u(t,s)=$
\begin{equation}\label{udeterminant}
\frac{1}{\beta} \left| \begin{matrix} 0 & H_k(t,\alpha)  & \cdots & H_{N-1}(t,a-1) \\
\nabla^{j_1}H_{\nu-1}(b,\rho(s)) & \nabla^{j_1}H_{k}(b,\alpha) & \cdots & \nabla^{j_1}H_{N-1}(b,a-1)\\
\nabla^{j_2} H_{\nu-1}(b,\rho(s)) & \nabla^{j_2}H_k(b, \alpha) & \cdots & \nabla^{j_2}H_{N-1}(b,a-1) \\
\vdots & \vdots & \ddots  & \vdots \\
\nabla^{j_{N-k}} H_{\nu-1}(b,\rho(s)) & \nabla^{j_{N-k}} H_k(b,\alpha)  & \cdots & \nabla^{j_{N-k}}H_{N-1}(b,a-1)\end{matrix} \right|, 
\end{equation}
for $(t,s) \in \N_{a-N+1}^b \times \N_{a+1}^b$; with $\beta:= \det D$,  where $D$ is given by \eqref{nonzeroN-kdet}; $v(t,s):= u(t,s) + H_{\nu-1}(t,\rho(s))$; and $\alpha:=a-N+k$.
 \end{theorem}
 \begin{proof}
Let $u(t,s)$ be given by \eqref{udeterminant}. By Theorem \ref{Dnotzero}, $\beta \not = 0$, so $u$ is well defined. Then, expanding $u(t,s)$ along the first row, for each fixed $s$, $u(t,s)$ is a linear combination of $H_{k}(t,a-N+k), H_{k+1}(t,a-N+k+1), \hdots, H_{N-1}(t,a-1)$. Hence, for each fixed $s$, $u(t,s)$ is a solution to $\nabla_{a*}^{\nu} x(t) = 0$.  Note that  $\nabla^i H_{p}(a-N+k,a-N+p)= 0$ for each $i \in \N_0^{k-1}$ and $p \in \N_k^{N-1}$, so we have $\nabla^i u(a-N+k,s) = 0$ for each $i \in \N_0^{k-1}$. Hence, $u(t,s)$ satisfies the boundary conditions at $t=a-N+k$ given in \eqref{ubvp}. 
 
 Next, define $z(t,s):=$
 \[ \frac{1}{\beta}\left| \begin{matrix} H_{\nu-1}(t,\rho(s)) & H_k(t,\alpha) & \cdots & H_{N-1}(t,a-1) \\
\nabla^{j_1}H_{\nu-1}(b,\rho(s)) & \nabla^{j_1}H_{k}(b,\alpha) & \cdots & \nabla^{j_1}H_{N-1}(b,a-1)\\
\nabla^{j_2} H_{\nu-1}(b,\rho(s)) & \nabla^{j_2}H_k(b, \alpha)   & \cdots & \nabla^{j_2}H_{N-1}(b,a-1) \\
\vdots & \vdots & \ddots & \vdots \\
\nabla^{j_{N-k}} H_{\nu-1}(b,\rho(s)) & \nabla^{j_{N-k}} H_k(b,\alpha)  & \cdots & \nabla^{j_{N-k}}H_{N-1}(b,a-1)\end{matrix} \right|, \] where $\alpha=a-N+k$. Expanding $z(t,s)$ along the first row, we have
$z(t,s) = $

\[\frac{1}{\beta} H_{\nu-1}(t,\rho(s)) \left| \begin{matrix}
\nabla^{j_1}H_{k}(b,\alpha) & \cdots & \nabla^{j_1}H_{N-1}(b,a-1)\\
 \nabla^{j_2}H_k(b, \alpha)   & \cdots & \nabla^{j_2}H_{N-1}(b,a-1) \\
\vdots  & \ddots & \vdots \\
 \nabla^{j_{N-k}} H_k(b,\alpha) & \cdots & \nabla^{j_{N-k}}H_{N-1}(b,a-1)\end{matrix} \right| \]
 
 \[ + \frac{1}{\beta}\left| \begin{matrix} 0 & H_k(t,\alpha) & \cdots & H_{N-1}(t,a-1) \\
\nabla^{j_1}H_{\nu-1}(b,\rho(s)) & \nabla^{j_1}H_{k}(b,\alpha) & \cdots & \nabla^{j_1}H_{N-1}(b,a-1)\\
\nabla^{j_2} H_{\nu-1}(b,\rho(s)) & \nabla^{j_2}H_k(b, \alpha)   & \cdots & \nabla^{j_2}H_{N-1}(b,a-1) \\
\vdots & \vdots & \ddots & \vdots \\
\nabla^{j_{N-k}} H_{\nu-1}(b,\rho(s)) & \nabla^{j_{N-k}} H_k(b,\alpha) & \cdots & \nabla^{j_{N-k}}H_{N-1}(b,a-1)\end{matrix} \right|. \]
Hence, we have $z(t,s) = H_{\nu-1}(t,\rho(s)) + u(t,s).$

Next, for $m \in \N_1^{N-k}$, $\nabla^{j_m} z(b,s)=$
\[ \frac{1}{\beta}\left| \begin{matrix} \nabla^{j_m}H_{\nu-1}(b,\rho(s)) & \nabla^{j_m}H_k(b,\alpha)  & \cdots & \nabla^{j_m}H_{N-1}(b,a-1) \\
\nabla^{j_1}H_{\nu-1}(b,\rho(s)) & \nabla^{j_1}H_{k}(b,\alpha)& \cdots & \nabla^{j_1}H_{N-1}(b,a-1)\\
\nabla^{j_2} H_{\nu-1}(b,\rho(s)) & \nabla^{j_2}H_k(b, \alpha) & \cdots & \nabla^{j_2}H_{N-1}(b,a-1) \\
\vdots & \vdots  & \ddots & \vdots \\
\nabla^{j_{N-k}} H_{\nu-1}(b,\rho(s)) & \nabla^{j_{N-k}} H_k(b,\alpha) & \cdots & \nabla^{j_{N-k}}H_{N-1}(b,a-1)\end{matrix} \right|. \] 
Hence,  $\nabla^{j_m} z(b,s)= 0$ for each $m \in \N_{1}^{N-k}$. Since $z(t,s) = H_{\nu-1}(t,\rho(s)) + u(t,s)$, this means for each $m \in \N_1^{N-k}$,
$\nabla^{j_m} u(b,s) = - \nabla^{j_m} H_{\nu-1}(b,\rho(s)).$
Therefore, we have that $u(t,s)$ satisfies the boundary conditions at $t = b$ in \eqref{ubvp}. Thus, the result follows by Theorem \ref{GreensfunctionthmkNk}.
 \end{proof}

 In the next theorem, we apply Theorems \ref{GreensfunctionthmkNk} and \ref{Gdeterminant} to the special case of the BVP \eqref{eq2nthgenRF}, which confirms \cite[Theorem 4.6]{jul}.

\begin{theorem} \label{nthordersolngenRF}
 Consider the nabla Caputo $(N-1, 1)$  BVP 
 \begin{equation}\label{eq2nthgenRF}
 \begin{cases}
 -\nabla_{a*}^{\nu} x(t) = h(t),  &t \in \N_{a+1}^{b} \\
  \nabla^i x(a-1)=0,  &i \in \N_0^{N-2} \\
 \nabla^j x(b) = 0,
 \end{cases}
 \end{equation} with $\nu > 1$, $N:= \Ceil{\nu}$, $j \in \N_0^{N-1}$ fixed, $b-a \in \N_{\max\{1,j\}}$, and $h: \N_{a+1}^{b} \to \R$. Then, $x: \N_{a-N+1}^b \to \R$ is a solution to the $(N-1, 1)$  BVP \eqref{eq2nthgenRF} if and only if $x(t)$ satisfies the integral equation
 \begin{equation} \label{inteqnthgenRF}
 x(t) = \int_a^b G_{\nu}(t,s)h(s) \nabla s, 
 \end{equation}
 for $t \in \N_{a-N+1}^b$, where $G_{\nu}: \N_{a-N+1}^b \times \N_{a+1}^b \to \R$ is given by
 \begin{equation} \label{greensnthgenRF}
 G_{\nu}(t,s)= \begin{cases} \frac{H_{N-1}(t,a-1)H_{\nu-j-1}(b,\rho(s))}{H_{N-j-1}(b,a-1)}, & t \leq \rho(s) \\
\frac{H_{N-1}(t,a-1)H_{\nu-j-1}(b,\rho(s))}{H_{N-j-1}(b,a-1)} - H_{\nu-1}(t,\rho(s)), & t \geq \rho(s). \end{cases}
  \end{equation}
  \end{theorem}

\section{Lyapunov Inequalities}


In this section, we will prove our main result involving Lyapunov inequalities in Theorem \ref{higherordersubsituting}. First, we give a theorem involving uniqueness of solutions to initial value problems and prove some important lemmas which we will use in the proof of Theorem \ref{higherordersubsituting}.

 \begin{theorem} \label{uniqueIVPs}
Let $\nu > 1$, $N:=\lceil \nu \rceil$, and $f: \N_{a+1} \to \R$. Then, the initial value problem
\begin{equation} \label{IVPeqn}
\begin{cases} 
\nabla_{a*}^{\nu} x(t) + q(t) x(t-1) = f(t), \,\ t \in \N_{a+1} \\
x(a-i) = A_i, \,\ i \in \N_0^{N-1}
\end{cases}
\end{equation}
has a unique solution defined on $\N_{a-N+1}$. 
\end{theorem}
\begin{proof}
By the initial conditions in \eqref{IVPeqn}, $x(t)$ is uniquely defined for $t \in \N_{a-N+1}^a$. Expanding the operator $\nabla_{a*}^{\nu}$ gives
\begin{align*}
\nabla_{a*}^{\nu} x(t) &= \nabla_a^{-(N-\nu)}\nabla^N x(t) \\
&= \sum\limits_{s=a+1}^t H_{N-\nu-1}(t, \rho(s)) \sum\limits_{i=0}^N (-1)^i {N \choose i} x(s-i).
\end{align*} Hence, the equation in \eqref{IVPeqn} is equivalent to
\begin{equation} \label{IVPeqnexpanded}
\sum\limits_{s=a+1}^t H_{N-\nu-1}(t, \rho(s)) \sum\limits_{i=0}^N (-1)^i {N \choose i} x(s-i) + q(t) x(t-1) = f(t), \,\ t \in \N_{a+1}.
\end{equation} The result follows by induction on $k$, letting $t = a+k$ in \eqref{IVPeqnexpanded}. 
\end{proof}

In the next lemma, we will show that if there is a nontrivial solution to \eqref{eqmth}, it is not identically zero on the domain $\N_a^{b-1}$.

\begin{lemma} \label{higherordersubx(s-1)issue}
Let $\nu > 1$ and suppose $x:\N_{a-N+1}^b \to \R$ is a solution to the equation
\begin{equation} \label{eqmth}
\nabla_{a*}^{\nu}x(t) + q(t) x(t-1) = 0, \,\ t \in \N_{a+1}^b,
\end{equation}
where $\nu > 1$ and $N:=\Ceil{\nu}$. Assume $b-a \in \N_{N-1}$. If $x(t) = 0$ for all $t \in \N_a^{b-1}$, then $x(t) \equiv 0$ on $\N_{a-N+1}^b$.
\end{lemma}
\begin{proof} Using the expanded form of  \eqref{eqmth} given by \eqref{IVPeqnexpanded}, and
assuming $x(t) = 0$ for $t \in \N_a^{b-1}$, it follows by induction on $k \in \N_1^{N-1}$ that
 \begin{equation} \label{zerosummations}
 \Sum_{i=0}^N (-1)^i {N \choose i } x(a+k-i) = \Sum_{i=k+1}^N (-1)^i{N \choose i} x(a+k-i) = 0
 \end{equation} holds for all $k \in \N_1^{N-1}$. 
Letting $k = N-1, N-2, \hdots, 1$ in \eqref{zerosummations} in the given order, we get $x(a-1)=x(a-2)=\cdots = x(a-N+1)=0$, respectively. Additionally, we have $x(a)= 0$, which means by the uniqueness of solutions to IVPs given in Theorem \ref{uniqueIVPs}, $x(t) \equiv 0$ on its entire domain $\N_{a-N+1}^b$.
\end{proof}

The proofs of the next three propositions are straightforward and hence omitted.

 \begin{proposition} \label{funcbds}
Let  $\alpha > -1$ and $s \in \N_{a}$. Then, the following hold:
\begin{enumerate}
 \item If $t \in \N_{\rho(s)}$, then $H_{\alpha}(t, \rho(s)) \geq 0$; if $t \in \N_s$, then $H_{\alpha}(t,\rho(s)) > 0$.

\item If $t \in \N_{\rho(s)}$ and $\alpha > 0$, then $H_{\alpha}(t, \rho(s))$ is a decreasing function of $s$; if $t \in \N_{s}$ and $-1 < \alpha < 0$, then $H_{\alpha}(t, \rho(s))$ is an increasing function of $s$.

 \item  If $t \in \N_{\rho(s)}$ and $\alpha \geq 0$, then $H_{\alpha}(t, \rho(s))$ is a nondecreasing function of $t$; if $\alpha > 0$ and $t \in \N_s$, then $H_{\alpha}(t,\rho(s))$ is an increasing function of $t$.  Also, if $t \in \N_{s+1}$ and $-1 < \alpha < 0$, then $H_{\alpha}(t, \rho(s))$ is a decreasing function of $t$.
 \end{enumerate}

\end{proposition}

\begin{proposition} \label{funcabs} Let $f,g$ be nonnegative real-valued functions on a set $S$. Moreover, assume $f$ and $g$ attain their maximum in $S$. Then, for each fixed $t \in S$, 
\begin{align*}
\left| f(t) - g(t)\right| &\leq \max\{ f(t), g(t)\} \leq \max\{ \max\limits_{t \in S} f(t), \max\limits_{t \in S} g(t) \}.
\end{align*}
\end{proposition}

 \begin{proposition} \label{taylormonineqs}
If $0 < \nu \leq \mu$, then $H_{\nu}(t,a) \leq H_{\mu}(t,a),$ for each fixed $t \in \N_a$.
\end{proposition}

Throughout this section, we let
\begin{align}\label{A}
A&:= \max\left\{\frac{H_{\gamma-1}(b, a)}{H_1(b,a-1)}H_2(b,a-1), H_{\gamma}(b,a)\right\}.
\end{align}  
   
 \begin{lemma} \label{intabsGbounds}
 Let $s \in \N_{a+1}^b$ and $1 < \gamma \leq 2$.  Then, for $j = 0$ in \eqref{greensnthgenRF},
 \begin{equation} \label{absGbd1}
\left|\int_{a-1}^t G_{\gamma}(\tau, s) \nabla \tau \right| \leq A \text{ and } \left|\int_{t}^b G_{\gamma}(\tau, s) \nabla \tau \right| \leq A,
 \end{equation} 
 and, for $j=1$ in \eqref{greensnthgenRF},
 \begin{equation} \label{absGbd1j1}
\left|\int_{a-1}^t G_{\gamma}(\tau, s) \nabla \tau \right| \leq H_2(b,a-1) \text{ and } 
\left|\int_{t}^b G_{\gamma}(\tau, s) \nabla \tau \right| \leq H_2(b, a-1),
 \end{equation}
where $G_{\gamma}$ is defined by \eqref{greensnthgenRF} with $N=2$.
 \end{lemma}
 \begin{proof}
  For $(t,s) \in \N_{a-1}^b \times \N_{a+1}^b$, by \eqref{greensnthgenRF},
 \begin{align*}
 \int_{a-1}^t G_{\gamma}(\tau,s) \nabla \tau &= \int_{a-1}^{\rho(s)} \frac{H_1(\tau, a-1)H_{\gamma-j-1}(b, \rho(s))}{H_{1-j}(b,a-1)} \nabla \tau\\ &+ \int_{\rho(s)}^t \left[\frac{H_1(\tau, a-1)H_{\gamma-j-1}(b, \rho(s))}{H_{1-j}(b,a-1)}  - H_{\gamma-1}(\tau, \rho(s)) \right] \nabla \tau \\
 &=    \frac{H_{\gamma-j-1}(b, \rho(s))}{H_{1-j}(b,a-1)}\int_{a-1}^{t} H_1(\tau, a-1) \nabla \tau \\
 & - \int_{\rho(s)}^t H_{\gamma-1}(\tau, \rho(s)) \nabla \tau. 
 \end{align*}By Theorem \ref{taylormonprops}, part (3), $ 
 \int_{a-1}^t H_1(\tau, a-1) \nabla \tau =  H_2(t, a-1).$
Next, for $t > \rho(s)$, 
$
\int_{\rho(s)}^t H_{\gamma-1}(\tau, \rho(s)) \nabla \tau 
= H_{\gamma}(t, \rho(s))$. 
Note that if $\rho(s) \geq t$,  $\int_{\rho(s)}^t H_{\gamma-1}(\tau, \rho(s)) \nabla \tau = 0$.
Thus, $
\int_{a-1}^t G_{\gamma}(\tau, s) \nabla \tau = \frac{H_{\gamma-j-1}(b, \rho(s))}{H_{1-j}(b,a-1)}H_2(t,a-1)-H_{\gamma}(t, \rho(s)),
$ so
\begin{equation}\label{intG}
\left|\int_{a-1}^t G_{\gamma}(\tau, s) \nabla \tau \right| = \left|\frac{H_{\gamma-j-1}(b, \rho(s))}{H_{1-j}(b,a-1)}H_2(t,a-1)-H_{\gamma}(t, \rho(s))\right|, 
\end{equation}
for $ t \in \N_{s}^b$, and 
\begin{equation}\label{intG2}
\left| \int_{a-1}^t G_{\gamma}(\tau, s) \nabla \tau \right|= \left|\frac{H_{\gamma-j-1}(b, \rho(s))}{H_{1-j}(b,a-1)}H_2(t,a-1)\right|,
\end{equation} for $t \in \N_{a-1}^{\rho(s)}.$

We will now examine the first term from the right hand side of \eqref{intG} for the case $j=0$. By Proposition \ref{funcbds}, parts (1)-(2), for $s \in \N_{a+1}^b$, 
\begin{equation} \label{ineq1}
0 \leq H_{\gamma-1}(b, \rho(s)) \leq H_{\gamma-1}(b,a),
\end{equation} and, for $t \in \N_{a-1}^b$,
\begin{equation}\label{ineq2}
0 \leq H_{2}(t,a-1) \leq H_2(b, a-1).
\end{equation}  
Hence, by \eqref{ineq1} and \eqref{ineq2},
\begin{equation}\label{term1bds}
0 \leq \frac{H_{\gamma-1}(b, \rho(s))}{H_1(b,a-1)}H_2(t,a-1)  \leq \frac{H_{\gamma-1}(b, a)}{H_1(b,a-1)}H_2(b,a-1).
\end{equation}

Now we consider the second term in \eqref{intG}. By Proposition \ref{funcbds}, it follows that 
\begin{equation}\label{term2bds}
0 \leq H_{\gamma}(t,\rho(s)) \leq H_{\gamma}(b,a).
\end{equation}
From \eqref{intG}, \eqref{intG2}, \eqref{term1bds}, \eqref{term2bds}, and Proposition \ref{funcabs}, for the case $j=0$, we obtain
$
\left|\int_{a-1}^t G_{\gamma}(\tau,s)\nabla \tau \right| \leq \max\left\{\frac{H_{\gamma-1}(b, a)}{H_1(b,a-1)}H_2(b,a-1), H_{\gamma}(b,a)\right\},$
 so the first inequality in \eqref{absGbd1} holds.

Consider the case $j=1$. Then, 
by Proposition \ref{funcbds}, parts (1)-(2), for $s \in \N_{a+1}^b$,
$
0 \leq H_{\gamma-2}(b, \rho(s)) \leq H_{\gamma-2}(b,\rho(b)) = 1.$
Therefore, 
\begin{equation} \label{ineqfocalcase}
0 \leq H_{\gamma-2}(b, \rho(s))H_2(t,a-1)  \leq H_2(b,a-1).
\end{equation}
Then, by \eqref{term2bds}, \eqref{ineqfocalcase}, and Proposition \ref{funcabs}, we obtain
$
\left|\int_{a-1}^t G_{\gamma}(\tau,s)\nabla \tau \right| \leq \max\left\{H_2(b,a-1), H_{\gamma}(b,a)\right\} .
$
Since $ 1<\gamma \leq 2$, we have $H_{\gamma}(b,a) \leq H_{2}(b,a) \leq H_2(b,a-1)$ by Proposition \ref{funcbds}, part (3) and Proposition \ref{taylormonineqs}, so in the case $j=1$, we have
$
\left|\int_{a-1}^t G_{\gamma}(\tau,s)\nabla \tau \right| \leq H_2(b,a-1). 
$
Thus, the first inequality in \eqref{absGbd1j1} holds.

Similarly, from \eqref{greensnthgenRF}, for $(t,s) \in \N_{a-1}^b \times \N_{a+1}^b$,  $\int_{t}^b G_{\gamma}(\tau,s) \nabla \tau =$
 \begin{align*}
 \frac{H_{\gamma-j-1}(b, \rho(s))}{H_{1-j}(b,a-1)} \int_{t}^{b}  H_{1}(\tau, a-1)\nabla \tau - \int_{\rho(s)}^b  H_{\gamma-1}(\tau, \rho(s))\nabla \tau . 
  \end{align*} Using arguments similar to the above, we obtain the second inequalities in \eqref{absGbd1} and \eqref{absGbd1j1}. \end{proof}

\begin{theorem} \label{higherordersubsituting} Let $q: \N_{a+1}^b \to \R$,  $\nu > 2$, and $N:= \Ceil{\nu}$. Assume $b-a \in \N_{N-1}$, and consider the BVP \eqref{eqmth},
\begin{equation} \label{eqmthbcs1}
\nabla^{N-2}x(a-1) = 0, \,\ \nabla^{N-2}x(b) = 0, \,\ \nabla^i x(c_i) = 0 \text{ for } i \in \N_0^{N-3},
\end{equation} 
where $c_i \in \{ a-1, b\}$, for $i \in \N_0^{N-3}$. 
Let $A$ be as defined in \eqref{A} with $\gamma = \nu-N+2$.
If the boundary value problem \eqref{eqmth}, \eqref{eqmthbcs1} has a nontrivial solution  $x: \N_{a-N+1}^b \to \R$, then 
\[ \int_a^b |q(s)| \nabla s \geq \frac{1}{A} \cdot \frac{1}{(b-a+1)^{N-2}}. \]

Furthermore, consider the BVP \eqref{eqmth}, 
\begin{equation} \label{eqmthbcs2}
\nabla^{N-2}x(a-1) = 0, \,\ \nabla^{N-1}x(b) = 0, \,\  \nabla^i x(c_i) = 0 \text{ for } i \in \N_0^{N-3},
\end{equation} 
where $c_i \in \{ a-1, b\}$, for $i \in \N_0^{N-3}$. 
If the boundary value problem \eqref{eqmth}, \eqref{eqmthbcs2} has a nontrivial solution $x: \N_{a-N+1}^b \to \R$, then 
\[ \int_a^b |q(s)| \nabla s \geq \frac{1}{H_2(b, a-1)} \cdot \frac{1}{(b-a+1)^{N-2}}.\]
\end{theorem}
\begin{proof}We will give the proof for the case of the BVP \eqref{eqmth}, \eqref{eqmthbcs1}. The proof is similar for the BVP \eqref{eqmth}, \eqref{eqmthbcs2}.
First, note that for $t \in \N_{a+1}$,
\begin{align*}
\nabla_{a*}^{\nu} x(t) &\overset{\eqref{fracdifform}}{=} \nabla_{a}^{-(N-\nu)}\nabla^N x(t) \\
&=\nabla_a^{-(N-\nu)}\nabla^2\nabla^{N-2}x(t) \\
&= \nabla_a^{-(2-(\nu-N+2))}\nabla^2\nabla^{N-2} x(t)\\
&= \nabla_{a*}^{\nu-N+2}\nabla^{N-2} x(t).
\end{align*} Hence, we can rewrite \eqref{eqmth} as
\begin{equation} \label{eqmthrewrite}
\nabla_{a*}^{\nu-N+2}\nabla^{N-2} x(t) + q(t) x(t-1) = 0, \,\ t \in \N_{a+1}^b.
\end{equation}
Let $y(t):= \nabla^{N-2} x(t)$ for $t \in \N_{a-N+3}^b$. Then, $y(t)$ solves the BVP
\begin{equation} \label{reducedtoconj}
\begin{cases}
-\nabla_{a*}^{\nu-N+2} y(t) = q(t) x(t-1), & t \in \N_{a+1}^b \\
y(a-1) = y(b) = 0.
\end{cases}
\end{equation}
By  Theorem \ref{nthordersolngenRF}, we have
$y(t) = \int_a^b G_{\nu-N+2}(t,s)q(s)x(s-1) \nabla s,$
for $t \in \N_{a-1}^b$,
where $G_{\nu-N+2}$ is given by \eqref{greensnthgenRF}  with $j= 0$.
Since $\nabla^{N-2} x(t) = y(t)$, we have
\begin{equation} \label{nablasquared}
\nabla^{N-2} x(t) = \int_a^b G_{\nu-N+2}(t,s)q(s)x(s-1) \nabla s.
\end{equation}

Applying Theorem \ref{ftnc}, with the appropriate boundary condition \\ $\nabla^{N-3} x(a-1) =0$ or $\nabla^{N-3} x(b) = 0$ given by \eqref{eqmthbcs1}, for $t \in \N_{a-1}^b$, we get either
\begin{align} \label{bcata-1}
\nabla^{N-3} x(t) &= \nabla^{N-3} x(t) - \nabla^{N-3} x(a-1) \nonumber\\
&\overset{\eqref{nablasquared}}{=} \int\limits_{a-1}^t \int\limits_a^b G_{\nu-N+2}(\tau, s) q(s) x(s-1) \nabla s \nabla \tau\nonumber \\
&= \int\limits_a^b q(s) x(s-1) \left(  \hspace{3pt} \int\limits_{a-1}^t G_{\nu-N+2}(\tau, s) \nabla \tau \right) \nabla s,
\end{align}
or \begin{align}\label{bcatb}
-\nabla^{N-3} x(t) &= \nabla^{N-3}x(b) -\nabla^{N-3} x(t)\nonumber  \\
&=\int\limits_a^b q(s) x(s-1) \left(  \int\limits_t^b G_{\nu-N+2}(\tau, s) \nabla \tau \right) \nabla s,
\end{align}
respectively, where we have interchanged the order of integration by using the linearity of the nabla integral. 
Let $F(t_0) :=\int_{a-1}^{t_0} G_{\nu-N+2}(\tau, s) \nabla \tau$ for the case \eqref{bcata-1} or $F(t_0):= \int_{t_0}^b G_{\nu-N+2}(\tau, s) \nabla \tau $ for the case \eqref{bcatb}. Assume, without loss of generality, that the remaining boundary conditions in \eqref{eqmthbcs1} are $\nabla^i x(a-1)=0$, for $i \in \N_0^{N-4}$.  Then, assuming $N \geq 4$, integrating either \eqref{bcata-1} or \eqref{bcatb} $N-3$ times and then taking the absolute value of both sides, we get in either case
\begin{align*} 
|x(t)| &= \left| \hspace{3pt} \int\limits_{a-1}^t \cdots \int\limits_{a-1}^t \int\limits_{a-1}^t \int\limits_a^b q(s) x(s-1) F(t_0) \nabla s \nabla t_0 \nabla t_1 \nabla t_2 \cdots \nabla t_{N-3} \right| \\
&\leq  \int\limits_{a-1}^t \cdots \int\limits_{a-1}^t \int\limits_{a-1}^t \int\limits_a^b |q(s)| |x(s-1)|\left|F(t_0)\right| \nabla s \nabla t_0 \nabla t_1 \nabla t_2 \cdots \nabla t_{N-3} \\
& \leq  \int\limits_{a-1}^b \cdots \int\limits_{a-1}^b \int\limits_{a-1}^b \int\limits_a^b |q(s)| |x(s-1)| \left|F(t_0)\right| \nabla s \nabla t_0 \nabla t_1 \nabla t_2 \cdots \nabla t_{N-3},
\end{align*}
for $t \in \N_{a-1}^b$.
Let $t = t'$ such that $x(t') = \max\limits_{t \in \N_{a-1}^b} |x(t)|$ and define $B:= x(t')$. By Lemma \ref{higherordersubx(s-1)issue}, since by  hypothesis $x$ is a nontrivial solution, we have $x(t) \not \equiv 0$ on $\N_a^{b-1}$. In particular, $B \not = 0$. Letting $t = t'$ in the last inequality, we get
\begin{align*}
|x(t')| 
& \leq \int\limits_{a-1}^b \cdots \int\limits_{a-1}^b \int\limits_{a-1}^b \int\limits_a^b |q(s)| |x(s-1)| \left|F(t_0)\right| \nabla s \nabla t_0 \nabla t_1 \nabla t_2 \cdots \nabla t_{N-3} \\
& \leq \int\limits_{a-1}^b \cdots \int\limits_{a-1}^b \int\limits_{a-1}^b \int\limits_a^b |q(s)| B \left|F(t_0)\right| \nabla s \nabla t_0 \nabla t_1 \nabla t_2 \cdots \nabla t_{N-3},  
\end{align*}
so 
$B \leq \int\limits_{a-1}^b \cdots \int\limits_{a-1}^b \int\limits_{a-1}^b \int\limits_a^b |q(s)| B \left|F(t_0)\right| \nabla s \nabla t_0\nabla t_1 \nabla t_2 \cdots \nabla t_{N-3}.$
 By Lemma \ref{intabsGbounds}, $|F(t_0)| \leq A$, where $A$ is defined by \eqref{A} with $\gamma = \nu - N +2$. Therefore, we get
$
1 \leq  \int\limits_{a-1}^b \cdots \int\limits_{a-1}^b \int\limits_{a-1}^b \int\limits_a^b |q(s)| A   \nabla s \nabla t_0 \nabla t_1 \nabla t_2 \cdots \nabla t_{N-3}.
$ It follows that
$ \frac{1}{A(b-a+1)^{N-2}}  \leq \int_a^b |q(s)| \nabla s. $

Note that for the proof involving the boundary conditions \eqref{eqmthbcs2}, we apply the bound $|F(t_0)| \leq H_2(t,a-1)$ given in Lemma \ref{intabsGbounds}.
\end{proof}

The proof of the next theorem is similar to the proof of Theorem \ref{existenceuniquenessbvpkNk}.

\begin{theorem} \label{nonhomsubstitution1} Let $q: \N_{a+1}^b \to \R$, $\nu > 2$, $N:= \Ceil{\nu}$, $b-a \in \N_{N-1}$, and $r \in \{1, 2\}$ be fixed. If the homogeneous BVP 
\begin{equation*} 
\begin{cases}
\nabla_{a*}^{\nu}x(t) + q(t) x(t-1) = 0, \,\ t \in \N_{a+1}^b \\
\nabla^{N-2}x(a-1) = 0, \,\ \nabla^{N-r}x(b) = 0, \,\ \nabla^i x(c_i) = 0, \text{ for } i \in \N_0^{N-3},
\end{cases}
\end{equation*} 
where $c_i \in \{ a-1, b\}$, for $i \in \N_0^{N-3}$, has only the trivial solution, then the nonhomogenous BVP
\begin{equation} \label{nonhomBVPhigherorder1}
\begin{cases} 
\nabla_{a*}^{\nu}x(t) + q(t) x(t-1) = f(t), \,\ t \in \N_{a+1}^b \\
\nabla^{N-2}x(a-1) = A_0, \,\ \nabla^{N-r}x(b) = B_0, \,\ \nabla^i x(c_i) = C_i, \text{ for } i \in \N_0^{N-3},\end{cases}
\end{equation}
where $f: \N_{a+1}^b \to \R$ and $A_0,$ $B_0,$ $C_i \in \R$ for $i \in \N_0^{N-3}$,
has a unique solution defined on $\N_{a-N+1}^b$.
\end{theorem}

Using the Lyapunov inequalities in Theorem \ref{higherordersubsituting} along with Theorem \ref{nonhomsubstitution1}, we get the following corollary.

\begin{corollary}  Let $q: \N_{a+1}^b \to \R$, $\nu > 2$, $N:= \Ceil{\nu}$, and $r \in \{1, 2\}$ be fixed. Assume $b-a \in \N_{N-1}$.  Consider the nonhomogeneous boundary value problem 
\eqref{nonhomBVPhigherorder1}.
\begin{enumerate}
 \item If $q(t)$ satisfies 
$ \int_a^b |q(t)| \nabla t < \frac{1}{A} \cdot \frac{1}{(b-a+1)^{N-2}}, $
then the BVP \eqref{nonhomBVPhigherorder1} with $r = 2$ has a unique solution defined on $\N_{a-N+1}^b$.

\item If $q(t)$ satisfies 
$\int_a^b |q(t)| \nabla t < \frac{1}{H_2(b, a-1)} \cdot \frac{1}{(b-a+1)^{N-2}}, $
then the BVP \eqref{nonhomBVPhigherorder1} with $r = 1$ has a unique solution defined on $\N_{a-N+1}^b$.
\end{enumerate}
\end{corollary}


\bibliographystyle{tfs}
\bibliography{dissertation}

\begin{thebibliography}{10}
\providecommand{\MR}{\relax\unskip\space MR }
\providecommand{\url}[1]{\normalfont{#1}}
\providecommand{\urlprefix}{Available at }

\bibitem{rightfi}
T. Abdeljawad, Q.M. Al-Mdallal, and M.A. Hajji, \emph{Arbitrary order
  fractional difference operators with discrete exponential kernels and
  applications}, Discrete Dynamics in Nature and Society 2017 (2017).
  \urlprefix\url{https://doi.org/10.1155/2017/4149320}.

\bibitem{fahri}
M.F. Aktas and  D.Cakmak, \emph{Lyapunov-type inequalities for third-order
  linear differential equations}, Electronic Journal of Differential Equations
  2017 no. 139 (2017), pp. 1--14.

\bibitem{greens}
A. Cabada, J.A. Cid, and B. M\'aquez-Villamarín, \emph{Computation of
  {G}reen's functions for boundary value problems with {M}athematica}, Applied
  Mathematics and Computation 219 no. 4 (2012), pp. 1919--1936.

\bibitem{dk1}
S. Dhar and Q. Kong, \emph{Fractional {L}yapunov-type inequalities with mixed
  boundary conditions on univariate and multivariate domains}, Journal of
  Mathematical Inequalities (Submitted)  (2017).

\bibitem{dk2}
S. Dhar and Q. Kong, \emph{Lyapunov-type inequalities for $\alpha$-th order
  fractional differential equations with $2 < \alpha \leq 3$ and fractional
  boundary conditions}, Electronic Journal of Differential Equations 2017 no.
  203 (2017), pp. 1--15.

\bibitem{f1}
R.A. Ferreira, \emph{On a {L}yapunov-type inequality and the zeros of a certain
  {M}ittag-{L}effler function}, Journal of Mathematical Analysis and
  Applications 412 no. 2 (2013), pp. 1058--1063.

\bibitem{f2}
R.A. Ferreira, \emph{Some discrete fractional {L}yapunov-type inequalities},
  Fractional Differential Calculus 5 no. 1 (2015), pp. 87--92.

\bibitem{sequentialFerr}
R.A. Ferreira, \emph{Lyapunov-type inequalities for some sequential fractional
  boundary value problems}, Advances in Dynamical Systems and Applications 11
  no. 1 (2016), pp. 33--43.

\bibitem{deltaLyap}
K. Ghanbari and Y. Gholami, \emph{New classes of {L}yapunov type inequalities
  of fractional {$\Delta$}-difference {S}turm-{L}iouville problems with
  applications}, Bulletin of the Iranian Mathematical Society 43 no. 2 (2017),
  pp. 385--408.

\bibitem{jul}
J.S. Goar, \emph{A {C}aputo boundary value problem in nabla fractional
  calculus}, Ph.D. diss., University of Nebraska-Lincoln,  2016.

\bibitem{gp}
C. Goodrich and A.C. Peterson, \emph{Discrete Fractional Calculus}, Springer,
  2015.

\bibitem{jrs}
M. Jleli, L. Ragoub, and B. Samet, \emph{A {L}yapunov-type inequality for a
  fractional differential equation under a {R}obin boundary condition}, Journal
  of Function Spaces 2015 (2015).
  \urlprefix\url{https://doi.org/10.1155/2015/468536}.

\bibitem{js}
M. Jleli and B. Samet, \emph{Lyapunov-type inequalities for fractional
  boundary-value problems}, Electronic Journal of Differential Equations 2015
  no. 88 (2015), pp. 1--11.

\bibitem{kp}
W.G. Kelley and A.C. Peterson, \emph{The Theory of Differential Equations},
  Springer (2010).

\bibitem{l}
A.M. Liapunov, \emph{Probleme general de la stabilite du mouvement}, Annals of
  Mathematical Studies 17 (1947), pp. 203--474.

\bibitem{ma}
D. Ma, \emph{A generalized {L}yapunov inequality for a higher order fractional
  boundary value problem}, Journal of Inequalities and Applications 2016 no.
  261 (2016). \urlprefix\url{https://doi.org/10.1186/s13660-016-1199-5}.

\bibitem{vandermonde}
L. Mirsky, \emph{An introduction to linear algebra}, Dover Publications, Inc.,
  1955.

\bibitem{pb}
A.C. Peterson and M. Bohner, \emph{Dynamic Equations on Time Scales},
  Birkhauser, 2001.

\end{thebibliography}
\bigskip

\appendix

\end{document}